\documentclass[11pt]{amsart}
\usepackage{graphicx}
\usepackage{pgf,tikz,pgfplots}
\usepackage{mathrsfs}
\usetikzlibrary{arrows}
\usepackage{mathpple}
\usepackage{float}
\usepackage{thmtools, thm-restate} 
\declaretheorem{theorem}

\usepackage{amsmath}
\usepackage{amssymb}
\usepackage{amsfonts}
\usepackage{amsthm}
\usepackage{appendix}
\usepackage{hyperref}
\usepackage{pst-node}
\usepackage{tikz-cd} 
\usepackage{quiver}
\usepackage{setspace}

\newcommand{\bun}{\mathrm{Bun}}
\newcommand{\B}{\mathrm{B}}
\newcommand{\D}{\mathrm{D}^b}
\newcommand{\map}{\mathrm{Map}}

\newcommand{\sll}{\mathrm{SL}_2}
\newcommand{\slr}{\mathrm{SL}_{n}}

\newcommand{\V}{\mathcal{V}}
\newcommand{\C}{\mathbf{C}}
\newcommand{\K}{\mathcal{K}}

\newcommand{\leng}{\ell}
\newcommand{\dmod}{\textrm{D-mod}}
\newcommand{\h}{h^\vee}

\newcommand{\sym}{\mathrm{Sym}}
\newcommand{\ind}{\mathrm{Ind\,}}
\newcommand{\indk}{\mathrm{Ind}_k}
\newcommand{\spec}{\mathrm{Spec\,}}
\newcommand{\Hom}{\mathrm{Hom}}
\newcommand{\ct}{\mathrm{CT}_*}
\newcommand{\cdl}{\mathcal{L}}
\newcommand{\co}{\mathcal{O}}

\newcommand{\N}{\mathbb{N}}

\newcommand{\g}{\mathfrak{g}}

\newcommand{\ga}{G^{ad}}

\newtheorem*{thm*}{Theorem}

\newtheorem{rmk}{Remark}
\newcommand{\ca}{\mathcal{A}}
\newtheorem{lem}{Lemma}
\newtheorem{cor}{Corollary}
\newtheorem*{cor*}{Corollary}

\newcommand{\rep}{\mathrm{Rep}}

\newcommand{\gr}{\mathrm{Gr}_G}
\newcommand{\grr}{\mathrm{Gr}_{G,\mathrm{Ran}}}
\newcommand{\ran}{\mathrm{Ran}}
\newcommand{\gm}{\mathbb{G}_m}

\newcommand{\loc}{\mathrm{Loc}}
\usepackage{graphicx}
\newcommand{\hk}{\mathrm{Hk}^\lambda}
\begin{document}

\title{Moduli of bundles and semiorthogonal decomposition}

\author{Kai Xu}
\email{kaixu@math.harvard.edu;kaixu1996@gmail.com}

\maketitle
\onehalfspacing
\begin{abstract}
    In this paper we construct semiorthogonal decompositions of moduli of principal bundles on a curve into its symmetric powers, for both the moduli stack of all $G$-bundles and the coarse moduli space of semistable $G$-bundles.  The essential ingredients in the proof include Borel-Weil-Bott theory for loop groups, highest weight structure of current group representation, derived $\Theta$-stratification and local-global compatibility of Kac-Moody localization. For the full moduli stack of all $G$-bundles the decomposition we construct is complete, while for the coarse moduli space of semistable $G$-bundles it is not clear in full generality yet.
\end{abstract}

\section{Introduction}

Fano varieties occupy a central place in birational geometry as part of the building blocks of general algebraic varieties. Semiorthogonal decompositions of the derived category of coherent sheaves on a Fano variety are deeply related to its birational geometry and motive \cite{kuznetsov2010derived}. In contrast to the Calabi-Yau and general type case, the derived category of a Fano variety is always decomposable, and its semiorthogonal decomposition is closely related to its motivic decomposition and rationality problem.

Moduli theory provides a rich source of higher-dimensional Fano varieties: let $X$ be a smooth projective curve of genus $g>1$, $L$  be a line bundle on $X$, then the coarse moduli space\footnote{We will always use Roman type for coarse moduli space and Italic type for moduli stack.} $Bun_{n}^{L,ss}$  of semistable vector bundles on $X$ with rank $n$ and determinant $L$, are Fano varieties of dimension $(g-1)(n^2-1)$ (smooth if $\delta=(n,\deg L)=1$ and with Gorenstein canonical singularity in general \cite{beauville1998picard}). Analogous construction exists for any semisimple algebraic groups $G$. In low rank examples, they are known to admit a motivic decomposition into products of symmetric powers $\sym^{\lambda_1} X\times\cdots \times \sym^{\lambda_{n-1}} X$. They are believed to lift to a semiorthogonal decomposition \cite{lee2018remarks,lee2021symmetric}, although we only have partial decompositions yet.. 



This is one of our main results:

\begin{restatable}{theorem}{vb}
     Let $X$ be a smooth projective curve of genus $g>1$, $L$  be a line bundle on $X$ so that $\delta=(n,\deg L)=1$. We have a semiorthogonal decomposition for ${Bun}_{n}^{L,ss}$:
\begin{equation*}
    \D({Bun}_{n}^{L,ss})=\langle \theta^{\otimes k}\otimes \D(\sym^{\lambda_1} X\times\cdots \times \sym^{\lambda_{n-1}} X),\cdots\rangle
\end{equation*}
where $\theta$ is the theta line bundle (positive generator of the Picard group), integers $(k,\lambda_1,\cdots,\lambda_{n-1})$ satisfy $0\leq k\leq 1, 0\leq \lambda_i$, $\sum_{i\leq l} (\lambda_i-(g-1))i(n-l)+\sum_{i>l} (\lambda_i-(g-1))l(n-i)\leq 0$ for all $0<l<n$, and when $k=1$ all other upper bounds are strict. $\cdots$ is the semiorthogonal complement of the previous subcategories. 
\end{restatable}

With our choice of Fourier-Mukai functor defined by the global Weyl modules, the order of blocks is that less dominant $\lambda$ maps to more dominant $\lambda$, and smaller $k$ maps to larger $k$. The derived categories and semiorthogonal decompositions are well-behaved (i.e. saturated) in the smooth case, and in the singular case, we can replace $\D(Bun_{n}^{L, ss})$ by a categorical/non-commutative resolution of singularity $\D(Bun_{n}^{L,ss})^\sim$ \footnote{It would be natural to compare with the quasi-BPS categories \cite{davison2023bps,puadurariu2025quasi}. For $G=SL_2$, the quasi-BPS category was completely understood in \cite{sink2025noncommutative}. We thank J. Tevelev for this comment.}. There are natural constructions from the rigidified moduli stack $\D(\bun_{n}^{L,ss}/\B\mu_{n})$ (and similarly for $\D(Bun_{n}^{d,ss})$). 

\begin{restatable}{theorem}{vbnc}
    Let $X$ be a smooth projective curve of genus $g>1$, $L$  be a line bundle on $X$, we have a non-commutative resolution of ${Bun}_{n}^{L,ss}$ and a semiorthogonal decomposition:
\begin{equation*}
    \D({Bun}_{n}^{L,ss})^\sim=\langle \theta^{\otimes k}\otimes \D(\sym^{\lambda_1} X\times\cdots \times \sym^{\lambda_{n-1}} X),\cdots\rangle
\end{equation*}
where $\theta$ is the theta line bundle (positive generator of the Picard group), integers $(k,\lambda_1,\cdots,\lambda_{n-1})$ satisfy $0\leq k\leq \delta$, $\delta$ divides $\sum i\lambda_i$, $\sum_{i\leq l} (\lambda_i-(g-1))i(n-l)+\sum_{i>l} (\lambda_i-(g-1))l(n-i)\leq 0$ for all $0<l<n$, and when $k=\delta$ all other upper bounds are strict. $\cdots$ is the semiorthogonal complement of the previous subcategories.
\end{restatable}

More general semiorthogonal decomposition exist for moduli of principal bundles: we can define a moduli space ${Bun}_{G}^{\xi,ss}$ of twisted $G$-bundles depending on a $Z(G)$-gerbe $\xi\in H^2(X,Z(G))$ which recovers ${Bun}_{n}^{L,ss}$ for $G=\slr$ and $\xi=\xi(L)$ the $n$-th root gerbe of $L$.

\begin{restatable}{theorem}{gnc}
 Let $X$ be a smooth projective curve of genus $g>1$, $G$ be a semisimple algebraic group of rank $r$, $Z(G)$ be the center of $G$ and $\xi\in H^2(X,Z(G))$, we have a semiorthogonal decomposition for ${Bun}_{G}^{\xi,ss}$:
\begin{equation*}
    \D({Bun}_{G}^{\xi,ss})^\sim=\langle \theta^{\otimes k}\otimes \D(\sym^{\lambda_1} X\times\cdots \times \sym^{\lambda_{r}} X),\cdots\rangle
\end{equation*}
where $\theta$ is the theta line bundle, integers $(k,\lambda_1,\cdots,\lambda_{r})$ satisfy $0\leq k<2 \delta$, $(g-1)\rho-\lambda$ lies in the cone generated by positive roots, and lies in the interior if $k=\delta$. Here $\rho$ is half the sum of positive roots, $\lambda=\sum_i\lambda_i\omega_i$ is the dominant weight with coefficients $\lambda_i$, $\delta=\frac{h^\vee}{\mathrm{ord}(\xi)}$, $h^\vee$ is the dual Coxeter number. $\cdots$ is the semiorthogonal complement of the previous subcategories, and $\D({Bun}_{r}^{L,ss})^\sim$ is a categorical resolution of singularity of ${Bun}_{G}^{\xi,ss}$ (which is $\D({Bun}_{r}^{L,ss})$ for smooth ${Bun}_{G}^{\xi,ss}$) 
\end{restatable}

We now record two simple examples explicitly: for $G=\sll$ the moduli is smooth only when $\deg L=2k+1$, when $G=\mathrm{SL}_3$ the moduli is smooth when $\deg L=3k\pm 1$ but they are isomorphic by conjugation.

\begin{cor}
        We have a semiorthogonal decomposition for ${Bun}_{2}^{L,ss}$ when $\mathrm{deg}\,L$ is odd where the last $\cdots$ is a phantom category and conjecturally $0$:
\begin{align*}
    \D({Bun}_{2}^{L,ss})=\langle &\D(pt),\D(X),\cdots,\D(\sym^{g-1}X),\\
    &\theta\otimes\D(pt),\theta\otimes\D(X),\cdots,\theta\otimes\D(\sym^{g-2}X),\cdots\rangle
\end{align*}
\end{cor}

This celebrated conjecture, raised independently by Belmans-Galkin-Mukhopadhyay and Narasimhan, was proved\footnote{None of the two papers prove $\cdots=0$ half of the conjecture, it was proved later by Tevelev in \cite{tevelev2023braid}.} independently by Tevelev-Torres\cite{tevelev2024bgmn} and Xu-Yau\cite{xu2021semiorthogonal}. See also \cite{lee2023derived} for following works. Our previous method does not seem to yield a complete decomposition for $r>3$, although it provides partial decompositions as discussed above.

\begin{cor}
        We have a semiorthogonal decomposition for ${Bun}_{3}^{L,ss}$ when $\mathrm{deg}\,L$ is not divisible by $3$ where $k_1+2k_2,2k_1+k_2\leq 3(g-1)$ and both inequality should be strict for components with the factor $\theta$:
\begin{align*}
    \D({Bun}_{3}^{L,ss})=\langle \D(pt),\theta\otimes \D(pt),\cdots,\D(\sym^{k_1}X\times\sym^{k_2}X),\\
    \theta\otimes\D(\sym^{k_1}X\times\sym^{k_2}X),\cdots,\D(\sym^{g-1}X\times\sym^{g-1}X),\cdots\rangle
\end{align*}
\end{cor}

Although we started with the Fano variety $Bun_n^{L,ss}$, it is more natural to view it as the coarse moduli space of the underlying stack. We have natural morphisms $Bun_n^{L,ss}\longleftarrow\bun_{n}^{L,ss}/\B\mu_{n}\longleftarrow \bun_{n}^{L,ss}\longrightarrow \bun_{n}^{L}$, where Roman type is used for moduli stacks and the map $\bun_{n}^{L,ss}/\B\mu_{n}\longrightarrow Bun_n^{L,ss}$ gives a categorical resolution of singularity. So the study of $\D(\bun_{n}^{L,ss})^\sim\simeq \D(\bun_{n}^{L,ss})^{\B\mu_{n}}$ is reduced to the study of $\D(\bun_{n}^{L,ss})$ and Theorem 2 is similarly reduced to the following:

\begin{restatable}{theorem}{vbs}
    We have a semiorthogonal decomposition for ${\bun}_{G}^{\xi,ss}$:
\begin{equation*}
    \D({\bun}_G^{\xi,ss})=\langle \mathcal{L}^{\otimes k}\otimes \D(\sym^{\lambda_1} X\times\cdots \times \sym^{\lambda_{n-1}} X),\cdots\rangle
\end{equation*}
where $\mathcal{L}$ is the fundamental determinant line bundle, $k\in \mathbb{Z}$ and $\lambda\in\N^{r}$ satisfy $0\leq k<2h^\vee$, $\sum_{i\leq l} (\lambda_i-(g-1))i(n-l)+\sum_{i>l} (\lambda_i-(g-1))l(n-i)\leq 0$ for all $0<l<n$, and when $k\geq h^\vee$ all other upper bounds are strict. $\cdots$ is the semiorthogonal complement of the previous subcategories. 
\end{restatable}

The open substack $\bun_G^{\xi,ss}\longrightarrow \bun_{G}^{\xi}$ is the semistable locus determined by the Harder-Narasimhan filtration, and the machinery of derived $\Theta$-stratification developed by Halpern-Leistner allows us to relate their derived categories. $\D(\bun_{G}^{\xi,ss})\longrightarrow \D(\bun_{G}^{\xi})$ is a full subcategory which satisfies certain restrictions of weights on the unstable loci. (This is how the finite subsets arise.) So Theorem 3 is reduced to the following:

    \begin{restatable}{theorem}{gs}
For a smooth projective curve $X$ we have the following semiorthogonal decomposition
\begin{equation*}
    \D(\bun_G^\xi)=\langle \mathcal{L} ^{\otimes k}\otimes \D(\sym^\lambda X) \rangle
\end{equation*}
where $\mathcal{L}$ is the fundamental determinant line bundle,  $k\in \mathbb{Z}$, $\lambda\in\Lambda^+\simeq\N^{r}$ are dominant weights, and $\sym^\lambda X=\prod \sym ^{\lambda_i}X$ are the corresponding variety of colored divisors.
\end{restatable}

This also has a non-projective analogue:

    \begin{restatable}{theorem}{gsaf}
For a smooth non-projective curve $X=\spec A$ we have the following semiorthogonal decomposition
\begin{equation*}
    \D(\bun_G)=\rep\, G(A)=\langle \D(\sym^\lambda X) \rangle
\end{equation*}
where $\lambda\in\Lambda^+\simeq\N^{r}$ are dominant weights, and $\sym^\lambda X=\prod \sym ^{\lambda_i} X=\bigotimes (A^{\otimes \lambda_i})^{S_i}$ are the corresponding variety of colored divisors.
\end{restatable}

The full moduli stack ${\bun}_{G}^{\xi}$ is much easier to study because of the Weil uniformization, which realizes ${\bun}_{G}^{\xi}$ as a double coset. This is most well-studied for trivial $\xi$, where we have $${\bun}_{G}^{\mathcal{O}}=\bun_G=\prod_{x\in X} G(\co_x)\backslash \prod_{x\in X} G(\K_x)/G(K_X)$$ Here $\co_x\simeq k[[t]]$, $\K_x\simeq k((t))$, and $K_X$ is the function field of $X$. We are taking the restricted product, the reader is referred to Section 2 for details. We prove that for nontrivial $\xi$ we need to twist the right-hand side by the outer-automorphism of $G(K_X)$ corresponding to $\xi$:
$${\bun}_{r}^L=\bun_G^\xi=\prod_{x\in X} G(\co_x)\backslash \prod_{x\in X} G(\K_x)/G(K_X)^\xi$$

This double coset representation allows us to reduce the study of tautological bundles and derived category to representation theoretical problems about loop group $G(\K_x)$ and current group $G(\co_x)$. $G(K_X)$ is negligible because of its homological contractibility. The representations of $G(\K_x)$ is clearly understood geometrically by the affine Borel-Weil-Bott theory, and Theorem 5 is finally reduced to the following:

\begin{restatable}{theorem}{loc}
    We have a semiorthogonal decomposition for $\B G(\co)$:
\begin{equation*}
    \D(\B G(\co))=\rep(G(\co))=\langle  \D(\sym^{\lambda_1} D\times\cdots \times \sym^{\lambda_{r}} D)\rangle
\end{equation*}
where $D=\mathrm{Spec}(\co)$, $\lambda_i\in\mathbb{N}$.
\end{restatable}

\begin{restatable}{theorem}{locc}
    We have a semiorthogonal decomposition for $\gr=G(\co)\backslash G(\mathcal{K})$:
\begin{equation*}
    \D(\gr)=\langle  \cdl^{\otimes k}\otimes \D(\sym^{\lambda_1} D\times\cdots \times \sym^{\lambda_{r}} D)\rangle
\end{equation*}
where $D=\mathrm{Spec}(\co)$, $k\in\mathbb{Z}$ and $\lambda_i\in\mathbb{N}$.
\end{restatable}

It is enlightening to notice that $\B G(\co)$ is the moduli of $G$ bundle on $D=\mathrm{Spec}(\co)$ while $\gr$ is the moduli of $G$ bundle on $D=\mathrm{Spec}(\co)$ equipped with a trivialization on $D^\circ=\mathrm{Spec}(\mathcal{K})$, thus Theorems 7,8 can be viewed as local analogues of Theorem 5 where we take $X=D$ with only ceveat that in the non compactly-supported setting the determinant line bundle it not defined. Theorem 7 follows from the highest weight structure of $\rep(G(\co))$ and Theorem 8 follows from the affine Borel-Weil-Bott theorem, the triviality of loop group representation at level $0$ and Theorem 6.

We now address the construction of functors between the two sides. A natural choice is to use tautological vector bundles as Fourier–Mukai kernels. This works particularly well for classical groups, say $\mathrm{SL}_{n}$. The tautological bundles naturally live on powers $X^n\times \bun_{n}^L$ and descend to the stacky quotient $X^n/S_n\times \bun_{n}^L$, but not on (products of) symmetric powers $\sym^n X$. A key observation is that $\sym^\lambda X=\sym^{\lambda_1} X\times\cdots \times \sym^{\lambda_{r}} X$ appears as a stacky stratum of $X^{|\lambda|}/S_{|\lambda|}$ where $|\lambda|=\sum i\lambda_i$, Polishchuk and Van den Bergh \cite{polishchuk2019semiorthogonal} shows this stratification gives a semiorthogonal decomposition of the quotient stack. Consequently the right-hand side of Theorem 7 also appears in $\D(D^n/S_n)$, and Theorem 7 furnishes a relation between $\rep(G(\co))$ and $\D(D^n/S_n)$. 

Let us first consider another local example $X=\mathbb{G}_m$, then the left hand side $\D({\bun}_{r}^{L})=\rep(G(\mathbb{G}_m))$ is the representation of (a version of) affine Lie group and $\D(\mathbb{G}_m^n/S_n)=\rep (S_n\ltimes \mathbb{Z}^n)$ is the representation of affine Weyl group. This parallel strongly suggests an affine Schur–Weyl duality connecting the two. Replacing $\mathbb{G}_m$ by $D$ leads to a close cousin of affine Schur-Weyl duality named Schur-Weyl duality for current algebras/groups. In the literature, these correspondences are typically established only at the abelian level and do not automatically lift to the derived setting as the representation category is not semisimple. Using the highest-weight structures on both sides, we prove a derived Schur-Weyl duality for current group which implies Theorem 7 in the $n\rightarrow \infty$ limit.

\begin{restatable}{theorem}{fm}
    The Fourier-Mukai transform associated with tautological bundle on $D^n/S_n\times \bun_{n}^L$ restricts to an embedding on the admissible subcategory of $\D(D^n/S_n)$ generated by $\D(\sym^{\lambda_1} D\times\cdots \times \sym^{\lambda_{r}} D)$ to $\D(\bun_{n}^L)=\rep \,G(\co)$ which exhaust the right hand side in the $n\rightarrow\infty$ limit.
\end{restatable}

This method was applied in an earlier version of the current paper \cite{xu2023moduli} to prove the main theorems for type A (and straightforwardly generalizes to other classical groups). The machinery of global Weyl modules \cite{khoroshkin2013highest,bennett2014macdonald,chari2015bgg} furnishes a parallel proof that avoids explicit use of Schur–Weyl duality and the stacky symmetric-power decomposition. These modules are the standard objects in the highest weight category $\rep(G(\co))$, whose highest weight structures are induced from $\rep\, G$. It is known that this construction is functorial, and we show that it can be glued from affine charts to a globally defined functor on products of symmetric powers. This is precisely the functor needed for the semiorthogonal decomposition; once it is in place, semiorthogonality follows from Borel–Weil–Bott for loop groups together with local calculations for global Weyl modules.

\subsection{Acknowledgements}

The author is grateful to S.-T. Yau for guidance and support over the years; to L. Chen and Y. Fu for teaching him geometric representation theory; to J. Tevelev for many helpful comments on the first draft, and to D. Auroux, R. Fujita, D. Halpern-Leistner, H.-B. Moon, and D. Ben-Zvi for helpful discussions. 



\section{Moduli of twisted bundles}

Throughout this paper, we fix an algebraically closed base field of characteristic $0$.\footnote{All of our constructions are canonical and Galois-equivariant, hence straightforwardly descend to non-closed fields.} Let $G$ be a simple algebraic group, $X$ be a smooth projective curve and $\bun_G=\map(X,\mathrm{B}G)$ be the moduli stack of $G$-bundles over $X$. Let $Z=Z$ be the center of $G$, $X,Y$ be the weight and root lattices, the $Z$ gerbes are classifield by $\map (X,\B^2Z)$, by Poincare duality this is isomorphic to $ \Hom( \map (X,\Hom(Z,\gm)),\mathbb{Q}/\mathbb{Z})$. Note that $\Hom(Z,\gm)\simeq Q$ is the quotient of root lattice of $G$ by weight lattice of $G$, we have a natural $\mathbb{Q}$-valued inner product on the weight lattice, such that root lattice is dual to weight lattice, so we get a $\mathbb{Q}/\mathbb{Z}$ valued bilinear pairing on their quotient $Q$ and this is a perfect pairing. Hence it induces a natural isomorphism $Q\simeq\Hom(Q,\mathbb{Q}/\mathbb{Z})$, and the $Z$ gerbes are classified by the homology of $X$ with coefficient $Q$, i.e. the isomorphism classes are classified by $Q$, the automorphisms are $H_1(X, Q)$ and the $2$-automorphism are $H_2(X, Q)$. Let $\ga=G/Z$ be the adjoint group of $G$, given a $Z$-gerbe $\xi$, we define the moduli of $\xi$-twisted $G$ bundle $\bun^\xi_G$ to be the fiber of $\map(X,\B\ga)\longrightarrow\map(X,\B^2Z)$ at $\xi$, which recover $\bun_G$ when $\xi$ is the trivial gerbe.\footnote{Note that $\map(X,\B^2Z)$ is a finite abelian group stack with nontrivial $\pi_0,\pi_1,\pi_2$ and we cannot replace it by $H^2(X,Z)$.}  Let $\bun_G^{\xi,s}$ (resp, $\bun_G^{\xi,ss}$) be the open substack of stable (resp, semistable) twisted bundles. For $G=\slr$, we have an equivalent description $\bun^\xi_G\simeq \bun_{n}^L$ as the moduli stack of rank $n$ vector bundles $E$ with an isomorphism $\det E \simeq L$ for some fixed $L$ whose degree equals to $\xi$ in $Q\simeq \mathbb{Z}/(r+1)\mathbb{Z}$.\footnote{Throughout the paper we use the étale topology on $X$.}. The equivalence follows from this Cartesian diagram of classifying stacks:

\[ \begin{tikzcd}
\B \mathrm{GL}_{n} \arrow{r} \arrow[swap]{d} & \B \mathbb{G}_m \arrow{d} \\%
\B \mathrm{PGL}_{n} \arrow{r}& \B^2 \mu_{n}
\end{tikzcd}
\]
as $\bun^\xi_G$ is the fiber of the bottom horizontal arrow after applying $\map (X,-)$ while $ \bun_{n}^L$ is the fiber of the top one. When $(n,\deg L)=1$, it's well known that the coarse moduli of stable bundles $Bun_{n}^{L,s}$ is a smooth projective variety and the moduli stack of stable bundles $\bun_{n}^{L,s}$ is a $Z$ gerbe over the coarse moduli $Bun_{n}^{L,s}$. 

Weil uniformization theorem plays a key role in the study of $\bun_G$. For our purpose we need to generalized it to $\bun_G^\xi$. Recall the (untwisted version of) Weil uniformization \cite{etde_20831639} states that 

\begin{restatable}{theorem}{weil}[Untwisted Weil uniformization]
    $$\bun_G=\prod_{x\in X} G(\co_x)\backslash \prod_{x\in X} G(\K_x)/G(K_X)$$
    where $\co_x$ means the completion of coordinate ring at $x$, $\K_x$ means its fraction field and $\prod$ means restricted product, i.e. all but finitely many factors are in $G(\co_x)$.
\end{restatable}  

In the twisted version of Weil uniformization we need to twist the action of $G(K_X)$ by an outer automorphism determined by $\xi$ (as we define later):
\begin{restatable}{theorem}{tweil}[Twisted Weil uniformization]
$$\bun_G^\xi=\prod_{x\in X} G(\co_x)\backslash \prod_{x\in X} G(\K_x)/G(K_X)^\xi$$
\end{restatable}

We reduce the twisted version to the usual one: we have $$\map(X,\B \ga)=\bun_{\ga}=\prod_{x\in X} \ga(\co_x)\backslash \prod_{x\in X} \ga(\K_x)/\ga(K_X)$$ We claim that there is a similar uniformization theorem for $\map(X,\B^2Z)$. First let us recall that classical Weil uniformization follows from the Drinfeld-Simpson theorem \cite{drinfeld1995b} that all $G$ bundles over curves are  Zariski-locally trivial. This local triviality also applies to $Z$ gerbes, as they are classified by the second cohomology group, which vanishes for affine curves. Thus we obtain the uniformization of $Z$ gerbes $$\map(X,\B^2Z)=\prod_{x\in X} \B Z(\co_x)\backslash \prod_{x\in X} \B Z(\K_x)/\B Z(K_X)$$
\begin{rmk}
    Technically speaking in the proof of uniformization we also need Beauville-Laszlo theorem for descent as the disc is not flat over a general (non-Noetherian) base ring $R$ and fpqc descent is not applicable. The corresponding fact for $\B^2 Z$ this is much easier as both sides of the uniformization are obviously of (ind-)finite type and we only need to prove the isomorphism on $R$ points for finite type algebra $R$, which follows form fpqc descent.
\end{rmk}

To prove the uniformization of $\bun^\xi_G$, we need the following general fact:

\begin{lem}
    Given the following diagram of group stacks:
    
\[ \begin{tikzcd}
H' \arrow{r} \arrow{d} & H  \arrow{r} \arrow{d}& H''\arrow{d}\\%
G' \arrow["f"]{r}  & G  \arrow["g"]{r} & G''\\%
K' \arrow{r} \arrow{u} & K  \arrow{r}\arrow{u} & K''\arrow{u}
\end{tikzcd}
\]
if the horizontal maps are fiber sequences, fiber $F_{g''}$ of the natural map $H\backslash G/K\longrightarrow H''\backslash G''/K''$ at a point $H''g''K''\in H''\backslash G''/K''$ is canonically isomorphic to $ H'\backslash G'/(K')^{g''}$.
\end{lem}

\begin{rmk}
    Here $g''$ acts by conjugation as an outer automorphism of $G'$ (which is canonically defined up to inner automorphism), so that $(K')^{g''}$ is a canonically defined subgroup of $G'$ up to conjugation and $ H'\backslash G'/(K')^{g''}$ is well-defined up to isomorphism. Moreover the map $ H'\backslash G'/(K')^{g''}\longrightarrow H'\backslash \mathrm{pt}=\B H'$ is also canonically defined up to conjugation by $h'\in H'$.
\end{rmk}
The proof is straightforward: choose a preimage $g$ of $g''$ in $G$, we then have a natural map $G\longrightarrow F_{g''}$, $x\mapsto H''xgK''$ which canonically descends to a map $ H'\backslash G'/(K')^{g''}\longrightarrow F_{g''}$. We may directly check that this is an isomorphism.

Applying this lemma to $\map(X,\B\ga)\longrightarrow\map(X,\B^2Z)$ we directly get the twisted version of Weil uniformization $$\bun_G^\xi=\prod_{x\in X} G(\co_x)\backslash \prod_{x\in X} G(\K_x)/G(K_X)^\xi$$

Note that the properness of $X$ is not used in the argument above, given $x_1,\cdots, x_n\in X$ we may apply the same argument to $\Sigma=X\backslash \{x_i\}$ and get $$\bun_G^\xi(\Sigma)=\prod_{x\in \Sigma} G(\co_x)\backslash \prod_{x\in \Sigma} G(\K_x)/G(K_X)^\xi$$
Since $G$ is simply-connected and semisimple, all $G$ bundles on $\Sigma$ are trivial by Drinfeld-Simpson \cite{drinfeld1995b} and $\bun_G^\xi(\Sigma)=\B  G(\Sigma)$, comparing the Borel-Weil-Bott isomorphism for $X$ and $\Sigma$ we obtain another form of uniformization theorem, whose untwisted version was proved in \cite{beauville1998picard}:
$$\bun_G^\xi=\prod_{i} G(\co_{x_i})\backslash \prod_{i} G(\K_{x_i})/ G(\Sigma)^\xi$$

Since the preceding proof is somewhat formal, we now expand it into more concrete terms: Drinfeld and Simpson showed that the untwisted version of uniformization follows from the Zariski-locally trivial property of $G$ bundles, and the same argument also applies to twisted bundles, as any $Z$ gerbe over $X$ is also Zariski-locally trivial. Following the same idea, for any $\xi$-twisted $G$ bundle $P$, we may first find a trivialization of $\xi$ over a Zariski open subset $U\subset X$ (i.e. complement of a finite set $S\subset X$), so that $P$ is given by an untwisted bundle over $U$; and then find a trivialization of $P_U$ over $U\subset X$ (we might need to shrink the $U$ in the first step). The data of a gerbe $\xi$ with trivialization on $U$ is equivalent to gerbes $\xi_x$ on discs $\mathcal{D}_x$ around $x\in S$ and trivializations on each punctured disc $\mathcal{D}_x^\circ$ (we may name them gerbes supported on $x$), such that after gluing they recover the original gerbe $\xi$ (in particular this implies that  $[\xi_x]\in H^2_c(\mathcal{D}_x,Z)\simeq Z$ satisfies their sum is $[\xi]\in  H^2_c(X,Z)\simeq Z(1)(k)^\vee$)

 After fixing the trivialization of $\xi$ the set of trivializations of $P_U$ forms a $G(\co_U)$ torsor (taking colimit over all $U$ gives a $G(\K_X)$ torsor). The data of a twisted bundle on $X$ with a trivialization on $U$ is equivalent to twisted bundles on discs $\mathcal{D}_x$ and trivializations on $\mathcal{D}_x^\circ$ (we may name them twisted bundles supported on $x$). In the untwisted case, the space of $G$ bundle supported on $x$ is by definition the fiber of 
 $\map(\mathcal{D}_x,\B G)\longrightarrow\map(\mathcal{D}_x^\circ,\B G) $ at the constant map and parameterized by affine Grassmannian $\mathrm{Gr}_G=G(\co_x)\backslash G(\K_x)$. In the twisted case, the space of $\xi_x$ twisted $G$ bundle supported on $x$ is the fiber of 
 $$\map(\mathcal{D}_x,\B\ga)\longrightarrow\map(\mathcal{D}_x^\circ,\B\ga)\underset{\map(\mathcal{D}_x^\circ,\B^2 Z)}{\times}\map(\mathcal{D}_x,\B^2 Z)$$ at the point given by $\xi_x$ and trivialization of $\xi_x$ at $\mathcal{D}_x^\circ$. One may expand the definition and show that this fiber space is canonically isomorphic to $G(\co_x)^{\xi_x}\backslash G(\K_x)$\footnote{It seems natural to name it twisted affine Grassmannian, but unfortunately this word already has other meaning.}. Hence the moduli of twisted bundles with trivialization on an open set is $\prod_{x\in X} G(\co_x)^{\xi_x}\backslash \prod_{x\in X} G(\K_x)$ and forgetting the trivialization we get $$\bun_G^\xi=\prod_{x\in X} G(\co_x)^{\xi_x}\backslash \prod_{x\in X} G(\K_x)/G(K_X)$$ which is the dual version of Weil uniformization where we swap $H$ and $K$ when applying the lemma. If we swap the roles of $U$ and $\mathcal{D}_x$ in the previous discussion, we will get the original version $$\bun_G^\xi=\prod_{x\in X} G(\co_x)\backslash \prod_{x\in X} G(\K_x)/G(K_X)^\xi$$

\section{Borel-Weil-Bott theory for $\bun_G^\xi$}

In this section we generalize the Borel-Weil-Bott theory for $\bun_G$ developed by Teleman \cite{teleman1998borel} to $\bun_G^\xi$. For application to semiorthogonal decomposition we only need the non-positive level case which is in some sense degenerate\footnote{Level $0$ case already exhibit rich interaction with current group representation.}, but we give the general result for its own interest and two independent proofs, one using Teleman's calculation of group cohomology and one using factorization homology method developed by Rozenblyum.

We have seen that $$\bun_G^\xi=\prod_{x\in X} G(\co_x)\backslash \prod_{x\in X} G(\K_x)/G(K_X)^\xi$$
hence we have a canonical map 
$$\bun_G^\xi\longrightarrow\prod_{x\in X} G(\co_x)\backslash\mathrm{pt}=\B \prod_{x\in X} G(\co_x)$$
along which we can pull back $\mathrm{Rep}\,\prod_{x\in X} G(\co_x)$ to $\D(\bun_G^\xi)$. We define tautological vector bundles as their tensor products with powers of the basic determinant line bundle $\cdl$ \cite{beauville1998picard} and denote $\D(\bun_G^\xi)^\mathrm{taut}$ to be the subcategory they generate. Since $\mathrm{Rep}\,\prod_{x\in X} G(\co_x)$ is generated by the image of the pull back map along natural projection $\mathrm{Rep}\,\prod_{x\in X} G(\co_x)\longrightarrow\mathrm{Rep}\,\prod_{x\in X} G_x$, it's natural to first study tautological bundles arising from $G$ representations. 

For tautological bundles defined by $G$ representations, the construction above can be reformulated as follows: given $N$ representations $V_i$ of $G$ with central character $\chi_i$, their external tensor product $\boxtimes V_i$ determines a coherent sheaf on $\B G^N$ and hence a twisted sheaf on $\B (G^{ad})^N$. We can pull it back along the natural evaluation map $X^{N}\times \bun^\xi_G\longrightarrow \B (G^{ad})^N$ and get a natural twisted sheaf on $X^{N}\times \bun^\xi_G$, which is only twisted in $X^N$ direction by the gerbe $\boxtimes (\chi_i\circ\xi)$. Fix a point $(x_i)\in X^N$ we may consider the fiber at $(x_i)$ which gives a vector bundle $\V$ on $\bun^\xi_G$ and this bundle is exactly the tautological bundle corresponding to $V=\boxtimes V_i\in \mathrm{Rep}\,\prod_{x\in X} G_x$. 

To state the Borel-Weil-Bott theorem, let us first set up our notations. Fix an integral level $k$, we call an irreducible representation $V$ of $G$ either regular or singular according to the location of the weight $(\lambda+\rho,\h+k)$ where $\lambda$ is the highest weight of $V$, $\rho$ is half-sum of the positive roots or equivalently sum of fundamental dominant weights, $\h$ is the dual Coxeter number of $G$ (which is $n$ for $\slr$) and $k$ is the level: for $k\geq 0$ we say it's regular if it lies in the interior of some Weyl alcove and define $\leng(V)$ to be the smallest length of affine Weyl transformation relating this Weyl alcove with the fundamental alcove, and we denote $V^{sm}$ to be the irreducible representation whose highest weight is the image of $\lambda$ in the fundamental Weyl alcove; otherwise when $(\lambda+\rho,\h+k)$ lies on the boundary of an affine Weyl alcove, we say this representation is singular. For $k<0$ we say all representations are singular because $(\lambda+\rho,\h+k)$ cannot be transformed into the fundamental alcove by an affine Weyl transformation.

Following is one form of the Borel-Weil-Bott theorem for $\bun_G^\xi$, which allows us to compute the cohomology of tautological bundles:
\begin{restatable}{theorem}{bwb}
$H^*(\bun^\xi_G,\cdl^{\otimes k}\otimes \V)$ vanishes if some $V_i$ is singular and is concentrated in degree $\sum\ell(V_i)$ if all of them are regular, in this case it's isomorphic to the global section $\Gamma(\bun^\xi_G,\cdl^{\otimes k}\otimes \V^{sm})$ where $\V^{sm}$ is the vector bundle obtained from representations $V_i^{sm}$.
\end{restatable}

The untwisted version of this theorem was proved for by Teleman in \cite{teleman1995lie,teleman1996verlinde,teleman1998borel}\footnote{In these papers, only $k\geq 0$ was explicitly studied, but as we will see $k<0$ case follows directly from the affine Borel-Weil-Bott theory.} using Lie algebra cohomology and later by Rozenblyum in \cite{rozenblyum2021connections} using factorization homology. Both of these approaches can be appropriately adapted to our twisted context.

By Weil uniformization theorem, we may compute this cohomology by representation of loop groups: choose $\infty\neq x_i\in X$, let $\Sigma=X\backslash \infty$, then we have the following $ \map(G,\Sigma)^\xi$ bundle

$$\mathrm{Gr}_G=G(\co_{\infty})\backslash  G(\K_\infty)\overset{p}{\longrightarrow} \bun_G^\xi= G(\co_{\infty})\backslash  G(\K_\infty)/\map(G,\Sigma)^\xi$$

The (derived) global section on the base may be computed by the (derived) $\map(G,\Sigma)^\xi$-invariant of global section on the total space, so that we have the following Grothendieck spectral sequence:
$$ H^*(\B  G(\Sigma)^\xi,H^*(\mathrm{Gr}_G,p^* (\cdl^{\otimes k}\otimes \V)))\Rightarrow H^*(\bun^\xi_G,\cdl^{\otimes k}\otimes \V)\,\,\,\, (\star)$$

We directly see that $p^* (\cdl^{\otimes k}\otimes \V)\simeq p^* (\cdl^{\otimes k})\otimes p^*(\V)$ and $p^*(\V)$ is the trivial vector bundle on $\gr$ with fiber $\otimes V_i$ carrying non-trivial $G(\Sigma)^\xi$ action as $\infty$ is disjoint from $\{x_i\}$. Now we need to study the basic line bundle $\cdl\in\mathrm{Pic}(\bun_G^\xi)$ more closely. To determine the Picard group of $\bun_G^\xi$ we use the descent spectral sequence for $\gm$ (which follows from the fact that $\gm$ as a sheaf is representable):
$$ H^*(\B G(\Sigma)^\xi,H^*(\mathrm{Gr}_G,\gm))\Rightarrow H^*(\bun^\xi_G,\gm)$$

In low degrees this gives an exact sequence
$$0\longrightarrow H^1(\B G(\Sigma)^\xi,\gm)\longrightarrow H^1(\bun^\xi_G,\gm)\longrightarrow  H^0(\B G(\Sigma)^\xi,H^1(\mathrm{Gr}_G,\gm))\longrightarrow H^2(\B G(\Sigma)^\xi,\gm)$$
where the last arrow maps $k\in \mathbb{Z}\simeq H^0(\B G(\Sigma)^\xi,H^1(\mathrm{Gr}_G,\gm))$ to the group cohomology class given by pulling back the level $k$ central extension of $G(\K)$ by the natural inclusion $ G(\Sigma)^\xi\longrightarrow G(\K)$ at $\infty\in X$. Here we need to recall that we have canonical isomorphism $\mathrm{Pic}(\gr)=H^1(\mathrm{Gr}_G,\gm)\simeq \mathbb{Z}$ on which $ G(\Sigma)^\xi$ acts trivially, hence we have $ H^0(\B G(\Sigma)^\xi,H^1(\mathrm{Gr}_G,\gm))\simeq\mathbb{Z}$. 
Also note that $ G(\Sigma)^\xi$ only differ from $ G(\Sigma)$ by a outer automorphism as a subgroup of $\prod_{x\in X} G(\K_x)$, hence there is an isomorphism canonical up to composition with inner automorphism. First group cohomology $H^1(\B G(\Sigma)^\xi,\gm)$ are group homomorphisms from $ G(\Sigma)^\xi$ to $\gm$ hence vanishes by the isomorphism above, and second group cohomology $H^2(\B G(\Sigma)^\xi,\gm)$ can be interpreted as central extensions of $ G(\Sigma)^\xi$ by $\gm$. Moreover, the canonical central extension of the loop group becomes trivial after pulling back along the natural inclusion $ G(\Sigma)^\xi\longrightarrow G(\K)$. On the Lie algebra level, the canonical central extension is determined by the Chevalley-Eilenberg cocycle $(f,g)\mapsto k\,\mathrm{Res}_\infty\langle f,g\rangle$, which is trivial as $\infty$ is the only possible pole, hence has residue $0$. However, the natural pull back from second group cohomology to Lie algebra cohomology is injective as $ G(\Sigma)$ is simply connected, hence the group central extension is also trivial.

Therefore the previous $\gm$ sequence is isomorphic to the following \footnote{One can prove that the last term has rank $2g$.}
$$0\longrightarrow H^1(\B G(\Sigma)^\xi,\gm)\simeq 0\longrightarrow H^1(\bun^\xi_G,\gm)\longrightarrow$$
$$ H^0(\B G(\Sigma)^\xi,H^1(\mathrm{Gr}_G,\gm))\simeq \mathbb{Z}\overset{0}{\longrightarrow} H^2(\B G(\Sigma)^\xi,\gm)$$
which gives the natural isomorphism $\mathrm{Pic}(\bun_G^\xi)\simeq\mathrm{Pic}(\gr)\simeq \mathbb{Z}$. The positive generator is our basic line bundle $\cdl$. We may directly check that the central character of $\cdl$ is exactly $\xi\in Q\simeq \Hom(Z,\gm)$. In particular, the order of the central character is exactly the order of $\xi$ in $Q$.

Now we see that the pull back of $\cdl$ also gives fundamental line bundle on the affine Grassmannian, whose cohomology is computed by the affine Borel-Weil-Bott theorem \cite{kumar2012kac}. More precisely, we have $H^*(\gr,p^* (\cdl^{\otimes k}))$ vanishes for $k<0$ and is one dimensional for $k=0$. For $k\geq 0$ this gives vacuum representation $H_c$ of loop groups at level $k$.

By the Grothendieck spectral sequence $(\star)$, we directly see that for $k<0$ we have $H^*(\bun^\xi_G,\cdl^{\otimes k}\otimes \V)\simeq 0$ and for $k\geq 0$ we have $H^*(\bun^\xi_G,\cdl^{\otimes k}\otimes \V)\simeq H^*(\B  G(\Sigma)^\xi,H_c\otimes V)$. Moreover the natural isomorphism $ G(\Sigma)^\xi\simeq G(\Sigma)$ commutes with the evaluation map $ G(\Sigma)^\xi\longrightarrow\prod_i G_{x_i}\times G_\infty$ as we may always choose two cycle representing $\xi$ supported away from $\{x_i\}\cup\{\infty\}$. Hence we have $H^*(\B  G(\Sigma)^\xi,H_c\otimes V)\simeq H^*(\B  G(\Sigma),H_c\otimes V)$, while the latter was computed by Teleman in \cite{teleman1998borel}. So we finish the proof of Borel-Weil-Bott theorem for twisted bundles.

Below we give the other proof by the machinery of factorization algebra. First we may observe that the restricted product in the Weil uniformization may be identified with the Beilinson-Drinfeld Grassmannian $\grr\overset{\pi}{\longrightarrow}\ran$ over the Ran space \cite{rozenblyum2021connections}. Corollary 7.3.4 loc. cit. states that the pulling back along the uniformization map $\grr\longrightarrow\bun_G$ is fully faithful. As noticed loc. cit. this is equivelent to the (coherent) homological contractibility of group of rational maps, i.e. $H^*(G(K_X),\co)$ is one-dimensional, by the isomorphism $G(K_X)\simeq G(K_X)^\xi$, the same holds for $G(K_X)^\xi$ and the same proof shows that $\grr\longrightarrow\bun_G^\xi$ is fully faithful. So $H^*(\bun^\xi_G,\cdl^{\otimes k}\otimes \V)\simeq H^*(\grr,p^* (\cdl^{\otimes k}\otimes \V))\simeq H^*(\ran,\pi_*p^* (\cdl^{\otimes k}\otimes \V))$. But this is the chiral homology of the Wess-Zumino-Witten chiral algebra $\ca_k$ whose !-fiber is $H_c=H^*(\gr,p^* (\cdl^{\otimes k}))$ with chiral modules at $x_i$ whose !-fiber is $H^*(\gr,p_i^* (\mathcal{V}_i))$, which is equivalent to Wess-Zumino-Witten chiral module as loop group representation by the affine Borel-Weil-Bott theorem. The computation of chiral homology for Wess-Zumino-Witten chiral algebra is well understood \cite{tuy}\footnote{They only consider the degree $0$ part but as Teleman shows higher degree vanish \cite{teleman1998borel} by a version of Kodaira vanishing theorem.}, which gives the Borel-Weil-Bott theorem we want. 

\begin{rmk}
    These two proofs are essentially independent, as the result we cited from \cite{teleman1998borel} in the first proof relies heavily on the structure of rational comformal field theory (more precisely Sugawara construction of projectively flat connection and the factorization of conformal blocks at degeneration); while the second approach only use the Wess-Zumino-Witten chiral algebra which only encodes the local geometric structure of the conformal field theory but not the global topological structure.
\end{rmk}

In this paper we will only use the $k\leq 0$ part of this theorem:
\begin{cor}
$H^*(\bun^\xi_G,\cdl^{\otimes k}\otimes \V)$ vanishes if $k<0$. For $c=0$ it also vanishes if for some of $V_i$ the highest weight $w_i$ lies outside the root lattice, otherwise it is one dimensional concentrated on degree $\langle\rho,\sum w_i\rangle$.
\end{cor}

It is straightforward to extend this calculation to the case when $\V$ is a tautological bundle from $G(\co)$ representations instead of $G$ representations. Fix a level $k$, we have natural induction functor $\indk:\rep\,G(\co)\longrightarrow \rep \,G(\mathcal{K})_k\simeq\rep\,\ca_k$. For $k<0$, $\ca_k=0$, $\rep\,\ca_k=\{0\}$ and hence $\indk$ is the zero functor; for $k=0$, $\ca_k$ is the unit factorization algebra, $\rep\,\ca_k=\mathrm{Vect}$, and $\indk$ is the (derived) functor taking $G(\co)$ invariants.

\begin{restatable}{theorem}{bwbg}
    Let $V_i$ be finite dimensional representations of current groups $G(\co_{x_i})$ and $\mathcal{V}$ be the corresponding tautological bundle pulled back along $\bun_G^\xi\longrightarrow\prod  G(\co_{x_i})$, $H^*(\bun^\xi_G,\cdl^{\otimes k}\otimes \V)$ is the factorization homology of $\ca_k$ with factorization module $\indk V_i$ inserted at $x_i$. In particular, it vanishes for $k<0$, and for $k=0$ it is given by group cohomology of $G(\co)$: $$H^*(\bun^\xi_G,\mathcal{V})=\bigotimes_i H^*(G(\co_{x_i}),V_i)$$
\end{restatable}

\section{Semiorthogonal decomposition of $\bun_G^\xi$}

In this section, we prove semiorthogonal decomposition  for the moduli stack $\bun_G^\xi$.

For the sake of completeness, we briefly review the theory of global Weyl modules of current groups \cite{khoroshkin2013highest,bennett2014macdonald,chari2015bgg}. Let us first recall the notion of highest weight categories, in particular the version introduced in \cite{feigin2022duality}: let $\C$ be a $k$-linear abelian category with enough projectives, and assume that simple objects $L_\C(\upsilon)$ are indexed by a set $\Upsilon$ with a surjection $\pi:\Upsilon\longrightarrow \Lambda$ onto a partially ordered set $\Lambda$ \footnote{In our example we have $\Upsilon\simeq \Lambda\times\mathbb{Z}$, and we denote $L_\C(\lambda)=L_\C(\lambda,0)$}. Then we have naturally defined subcategories $\C^{\leq \lambda}$ and $\C^{<\lambda}$ of $\C$ as well as the subquotient $\C^{\lambda}=\C^{\leq \lambda}/\C^{< \lambda}$ as generated by corresponding $L_\C(\upsilon)$. $\C$ is called a \textit{highest weight category} if the fully faithful embedding $\C^{\leq \lambda}\longrightarrow \C$ induces fully faithful embedding of derived categories $\D(\C^{\leq \lambda})\longrightarrow \D(\C)$. Then all of them are admissible subcategories, and their left orthogonal complements are $\D(\C^{> \lambda})$. More generally, given any subset $S\subset\Lambda$ which satisfies that for any $x\in S,\,y<x$ we have $y\in S$, we have full subcategory $\C^S\longrightarrow \C$ which induce fully faithful embedding on the derived category, and another full subcategory $\C^{\Bar{S}}\longrightarrow \C$ which is its right orthogonal complement. They are also highest weight categories in a natural way. We define $\Delta_\C(\lambda)$ to be the projective cover of $L_\C(\lambda)$ in $\C^{\leq \lambda}$, then we have a semiorthogonal decomposition $\D(\C^{\leq \lambda})=\langle \D(\C^{< \lambda}), \D(\langle\Delta_\C(\lambda\rangle)\rangle$, where by $\langle\Delta_\C(\lambda\rangle$ we mean the abelian subcategory it generates. This also implies $\C^{\leq \lambda}$ has a torsion pair $(\C^{< \lambda},\langle\Delta_\C(\lambda\rangle)$ decomposition. 

It is also convenient to define an abelian category $\C$ with distinguished generators $\Delta_\C(\lambda)$ to be a \textit{generalized highest weight category} if $\mathrm{Ext}^*(\Delta_\C(\lambda),\Delta_\C(\mu))=0$ unless $\lambda\leq \mu$ and $\mathrm{Ext}^*(\Delta_\C(\lambda),\Delta_\C(\lambda))$ is concentrated in degree $0$. Then we have a semiorthogonal decomposition $\D(\C)=\langle \D(\mathrm{End}(\Delta_\C(\lambda))) \rangle$. This is a highest-weight category in the usual sense if the algebras $\mathrm{End}(\Delta_\C(\lambda)))$ are semisimple.

It is a celebrated theorem ( \cite{khoroshkin2013highest,bennett2014macdonald,chari2015bgg} and also Example 1.12 in \cite{feigin2022duality}) that there is a highest weight structure on the category of graded representations of the current group $\rep^{\mathrm{gr}}(G[t])$ for any reductive group $G$ with $\Upsilon\simeq P_r\times\mathbb{Z}$ where $P_r$ is the partially ordered set of dominant weights and $\mathbb{Z}$ corresponds to the grading. The standard objects $\Delta(\lambda)$ are global Weyl modules\footnote{Our functorial definition being equivalent to the standard definition by generators and relations in current algebra literature follows from the fact that the current algebra is generated by nilpotent elements (cf. \cite{steinberg1962generateurs},xxx sorger), so that all representations whose nilpotent element acts nilpotently (including the global Weyl module) are automatically integrable to the current group.} $W(\lambda):=\mathrm{Ind}_{H\ltimes N[t]}^{G[t]}(k_\lambda)$ where $N[t]$ acts trivially on $k_\lambda$ and $H$ acts by the character $\lambda$. It follows that this gives the (abelian) category of ungraded representation a generalized highest-weight structure, and their endomorphism is explicitly known \cite{bennett2011homomorphisms}. This already gives us Theorem 6 for current group representations:

\loc*

We now proceed to the compactly supported analogue of this local decomposition:

\locc*

We first prove the pull back functor $\D(\B G(\co))\longrightarrow\D(\gr)$ is fully faithful. We need to calculate the cohomology of tautological bundles on the affine Grassmannian. This is equal to the induction functor $\rep\, G(\co)\longrightarrow\rep\,G(\mathcal{K})$ at level $k=0$, which is exactly the group cohomology, i.e. cohomology on $\B G(\co)$. This gives the semiorthogonal relations between components of the same level, and the affine Borel-Weil-Bott theorem gives the semiorthogonal relations between components of different levels. We then prove that the right hand side generates the left hand side, this is in fact a general fact for Kac-Moody flag variety. We refer readers to \cite{kumar2012kac,marquis2018introduction} for the basic facts about Kac-Moody groups. Note that $\gr$ is the partial flag variety of the affine Kac-Moody group $\mathbb{G}=\gm\ltimes \widehat{G(\mathcal{K})}$ quotient by its parabolic subgroup $\mathbb{P}=\gm\ltimes G(\co)\times \gm$, and the right hand side of Theorem 7 are exactly the image of $\rep\,\mathbb{P}$ (the first $\gm\subset \mathbb{P}$ does not produce nontrivial bundle on $\gr$). We only need to show that the pull back $\rep\,\mathbb{P}\simeq\D(\B\mathbb{P})\longrightarrow \D(\mathbb{G}/\mathbb{P})$ generates the target. Kac and Peterson \cite{kac1983regular,kac1985defining} (also c.f. \cite{tits2006groups}) proved that the Kac-Moody groups are (infinite-dimensional ind) affine algebraic groups satisfying a generalization of the Peter-Weyl theorem, which particularly implies that the coordinate ring $\Gamma(\mathbb{G},\co)$ viewed as $\mathbb{P}$ representation lies in $\mathrm{Pro(}\rep\,\mathbb{P})$ (i.e. is generated by the finite-dimensional representations). In other words, we have a pro-algebra $A=\Gamma(\mathbb{G},\co)\in\mathrm{Pro(}\rep\,\mathbb{P})$, whose specturm is $\mathbb{G}$ with compatible $\mathbb{P}$ action, and the relative spectrum functor over $\B\mathbb{P}$ gives $\mathbb{G}/\mathbb{P}\longrightarrow\B\mathbb{P}$. In particular, the pull back of $A=\Gamma(\mathbb{G},\co)\in\rep\,\mathbb{P}$ to $\mathbb{G}/\mathbb{P}$ generates $\D(\mathbb{G}/\mathbb{P})$\footnote{Here it's crucial to consider the bounded derived category $\D$ instead of $\mathrm{D}$ as the natural t-structure on the full derived category of quasi-coherent sheaves on infinite dimensional ind-schemes say $\mathbb{A}^\infty=\mathrm{colim}_n\, \mathbb{A}^n$ is typically not complete, and the global section of objects in degree $\infty$ vanish. Also, by restricting to the bounded part we have neglected the difference between quasi-coherent sheaves and ind-coherent sheaves.} , hence we see that the right hand side of Theorem 7 generates $\D(\gr)$.




To globalize, we need to observe that the definition of global Weyl modules still makes sense if we replace $\g[t]$ by $g\otimes A$ for any commutative algebra $A$ (in our case the coordinate ring of an affine chart on the curve) and is functorial with respect to localization. For more general $A$ it is also studied in \cite{feigin2004multi,chari2010categorical} (see \cite{braverman2014weyl,kato2018demazure} for equivalent definitions as cohomology of bundles on semi-infinite flag variety, which also globalize in a straightforward way): 

More formally, given an affine curve $\mathrm{Spec\,} A$, we may define the global weyl module
$$W_A(\lambda):=\mathrm{Ind}_{H\ltimes N(A)}^{G(A)}(k_\lambda)$$

As observed in \cite{chari2010categorical}, this is a $(A_\lambda, g\otimes A)$ bimodule, where $A_\lambda=\bigotimes (A^{\otimes \lambda_i})^{S_{\lambda_i}}$ so that $$\mathrm{Spec\,}A_\lambda= \sym^{\lambda_1}(\mathrm{Spec\,}A)\times\cdots \times \sym^{\lambda_r}(\mathrm{Spec\,}A)=\sym^\lambda (\mathrm{Spec\,}A)$$

where the $A_\lambda$ action is induced from the action of $H(A)\subset \mathrm{N}_{H\ltimes N(A)}^{G(A)}$, so that $W_A(\lambda)$ naturally lies in $ A_\lambda\,\mathrm{mod}\otimes \rep (G(A))$.

Given a Zariski localization $A\longrightarrow A[f^{-1}]$, we have a natural injection $W_A(\lambda)\longrightarrow W_{A[f^{-1}]}(\lambda)$, which naturally factors through the localization of modules along the ring map $A_\lambda\longrightarrow A[f^{-1}]_\lambda$ and get an injection $W_A(\lambda)\otimes_{A_\lambda}A[f^{-1}]_\lambda\longrightarrow W_{A[f^{-1}]}(\lambda)$. By the character calculation in \cite{chari2010categorical} each side has equal dimension on each stalk, hence this is an isomorphism, i.e. the functor $W_{(-)}(\lambda)$ maps localization of rings to localization of modules. Same argument works for formal completion at a point $A\longrightarrow \widehat{A}_x $. Note that we have a natural functor $\rep (G(A))\longrightarrow\D(\bun_G^\xi)$ from the restriction map $\bun_G^\xi(X)\longrightarrow\bun_G^\xi(\mathrm{Spec\,}A)=\B G(A)$ for any trivialization of $\lambda\circ\xi$ on $\mathrm{Spec\,}A$, where we used the fact that all $G$ bundles are trivial on each chart $\mathrm{Spec\,}A$ by Drinfeld-Simpson \cite{drinfeld1995b}. Also note that for any curve $X$, $\sym^\lambda X$ is covered by $\sym^\lambda U$ for affine charts $U$, i.e. $\sym^\lambda X=\mathrm{colim\,}\sym^\lambda U$. Hence the global Weyl module for affine charts canonically glue to a global Weyl module $W_X(\lambda)\in \D(\sym^\lambda X\times \bun_G^\xi)$. Note that $W_X(\lambda)$ depends on a global trivialization of $\lambda\circ\xi$, which always exists as $H^2(X,\gm)=0$, but in general not unique. A more canonical way is to define $W_X(\lambda)\in \D(\sym^\lambda X\times \bun_G^\xi,\lambda\circ \xi)$ as a twisted vector bundle, but we will omit the (trivial but not canonically trivialized) twisting unless necessary. It is convenient to view the disjoint union of $\sym^\lambda X$ as the moduli space of divisors in $X$ colored by the dominant weights of $G$, any colored divisor $\mathbf{x}=\sum \lambda_ax_a$ can naturally be viewed as a point in $ \sym^\lambda X$ where $\lambda=\lambda(\mathbf{x})$ is defined to be $\sum \lambda_a$.

The Fourier-Mukai transforms $W_\lambda:\D(\sym^\lambda X)\longrightarrow\D(\bun_G^\xi)$ associated with kernel $W_X(\lambda)\in \D(\sym^\lambda X\times \bun_G^\xi)$ are the functors we need in our semiorthogonal decomposition. The tensor product result in section 5.3 of \cite{chari2010categorical} implies that away from the diagonal of $\sym^\lambda X$ the image of $W_\lambda$ factorizes, more precisely, for two disjoint colored divisors $\mathbf{y},\mathbf{z}$ we have $$W_{\lambda(\mathbf{y}+\mathbf{z})}(\co_{\mathbf{y}+\mathbf{z}})=W_{\lambda(\mathbf{y})}(\co_{\mathbf{y}})\otimes W_{\lambda(\mathbf{z})}(\co_{\mathbf{z}})$$

So we only need to understand $W_{\lambda(\mathbf{x})}(\co_{\mathbf{x}})$ for $\mathbf{x}=\lambda x$ on the small diagonal of $\sym^\lambda X$, which corresponds to the local Weyl module on affine neiborhood $\mathrm{Spec\,}A$ of $x$ defined as follows:

$$W_A^{\mathrm{loc}}(\lambda):=\mathrm{Ind}_{B(A)}^{G(A)}(k_\lambda)$$

where $B(A)$ action on $V_\lambda$ factors through projection $B\longrightarrow B/N=H$ and evaluation at $x$. This is a finite-dimensional representation of $G(A)$ which factors through the map $G(A)\longrightarrow G(\co_x)$ (because Weyl module is compatible with completion), and as a $G$ representation it has highest weight $\lambda$ with one-dimensional weight space, whose character is calculated in \cite{chari2010categorical} and the reference therein.

Now we obtain our main theorem:

\gs*

\begin{proof}
    Now that we have defined all the functors, we only need to calculate the space of homomorphisms:

    $$\Hom^*_{\D(\bun_G^\xi)}(\mathcal{L} ^{\otimes k}\otimes W_{\lambda(\mathbf{y})}(\co_{\mathbf{y}}), \mathcal{L} ^{\otimes h}\otimes W_{\lambda(\mathbf{z})}(\co_{\mathbf{z}}))$$

    By the factorization property, this is equal to 
$$H^*({\bun_G^\xi}, \mathcal{L} ^{\otimes (h-k)}\otimes (\bigotimes_i(W_{\mu_i}(\co_{\mathbf{y}_i})^\vee \otimes W_{\nu_i}(\co_{\mathbf{z}_i}))))$$ 
where we write $\mathbf{y}=\sum \mathbf{y}_i$, $\mathbf{z}=\sum \mathbf{z}_i$, so that $\mathbf{y}_i=\mu_i x_i$ and $\mathbf{z}_i=\nu_i x_i$ for some points $x_i$ and dominant weights $\mu_i,\nu_i$ (which could possibly be $0$).

Now we can apply our generalized Borel-Weil-Bott theory and highest weight structure of current group representation to see that this is always $0$ if $h<k$ or $h=k$ and $\lambda(\mathbf{y})=\sum\mu_i\not\leq \lambda(\mathbf{z})=\sum \nu_i$ (so that $\mu_i\not \leq \nu_i$ for some $i$). Hence we get the semiorthogonal decomposition. 

To prove the complement of tautological bundles vanishes, again we need Weil uniformization and homological contractibility of $G(K_X)$. We write $\bun_G^\xi=\prod \mathrm{Gr}_{G,x}/G(K_X)^\xi$, as we have seen in Theorem 7 the tautological bundles from $\rep\, G(\co)\longrightarrow \D(\gr)$ together with the determinant line bundle $\cdl$ generates the derived category of the affine grassmannian. These generators are all eigensheaves with respect to the $G(K_X)^\xi$ action, so the subset of generators that are fixed by the $G(K_X)^\xi$ (i.e. has trivial eigenvalue $G(K_X)^\xi\longrightarrow \B\gm$) form a set of generators for $\bun_G^\xi$.  They are exactly the bundles such that the level $k_x$ on each $\mathrm{Gr}_{G,x}$ agree, i.e., pull back from a power of basic line bundle $\cdl^{\otimes k}$. So we see the generation by tautological bundles.
\end{proof}

For a non-projective curve $X$, we can either run the same proof or use its natural compactification to get the analogous theorem:

\gsaf*

In the remaining part of this section, we study the Hecke action on $\D(\bun_G)$ and its relation with the global Demazure modules. This part is independent of the rest of the paper.

Affine Demazure modules are defined as sections of tautological vector bundles on the affine Schubert varieties, and similar to the Weyl module, it has a global analogue defined using the Beilinson-Drinfeld Grassmannian, which was studied in \cite{dumanski2021beilinson} for $X=\mathbb{A}^1$. Instead of $\grr$ over Ran space, we consider its pull-back $\mathrm{Gr}_{G,\sym^\lambda X}$ along $\sym^\lambda X\longrightarrow\ran$ and its locally closed Schubert cell $\gr^\lambda\subset \mathrm{Gr}_{G,\sym^\lambda X}$ (we refer the readers to \cite{zhu2016introduction,dumanski2021beilinson} for detailed definitions), its closure $\overline{\gr^\lambda}$ is called affine Schubert varieties. Note that $\mathrm{Gr}_{G,\sym^\lambda X}$ admits natural maps to $\sym^\lambda X$ and to $\bun_G$, hence so does $\overline{\gr^\lambda}\subset \mathrm{Gr}_{G,\sym^\lambda X}$, this gives a (finite dimensional) correspondence between $\sym^\lambda X$ and $\bun_G$ and an interesting Fourier-Mukai functor. The fibers of this correspondence are natural global analogue of the slices of affine Grassmannian, which carry interesting symplectic singularities.

More generally, we want to consider the Hecke stack $\hk\longrightarrow \sym^\lambda X\times \bun_G\times \bun_G$, whose fibers over $(x_i,P_1,P_2)$ are isomorphisms between $P_1,P_2$ defined outside $x_i$ such that the poles at $x_i$ are bounded by the corresponding weight. If we take $P_1$ to be the trivial $G$-bundle, this recovers exactly $\overline{\gr^\lambda}$, and other fibers are exactly twisted versions of this construction. Hence, this correspondence defines a functor $\D(\sym^\lambda X)\otimes \D(\bun_G)\longrightarrow \D(\bun_G)$, and we may directly check that this gives an action of $\bigoplus_\lambda\D(  \sym^\lambda X)$ on $\D(\bun_G)$. $-\otimes \cdl^{\otimes k}$ gives a functor $\D(  \sym^\lambda X)\longrightarrow \D(\bun_G)$ whose Fourier-Mukai kernel is the global Demazure module $W_\lambda^k$. For an affine curve $X$ we may consider its compactification, and then our definition agrees with the one in \cite{dumanski2021beilinson}.

When $X=\mathbb{A}^1$, it is known in \cite{kato2022higher} that for $k\geq 1$ $W_\lambda^{k+1}$ admits a filtration by $W_\mu^{k}$ for $\mu\leq\lambda$ (more precisely $\mu$ should be an extremal weight of the level-k integrable highest weight representation of affine Lie algebra whose highest weight is $\lambda$). Moreover, if we define $W_\lambda^0=W_\lambda$ as the global Weyl module defined before, and $W_\lambda^{-k}=\Hom_{\sym^\lambda \mathbb{A}^1}(W_{\lambda^*}^{k},\co_{\sym^\lambda \mathbb{A}^1})$, $W_\lambda^{k+1}$ admits a filtration by $W_\mu^{k}$ for $\mu\leq\lambda$ for all $k$, this also holds for all integers $k$. Also in \cite{kato2022higher} it is proved that for each $k$ this gives rise to a highest weight structure of the representation of the current group $G[t]$. Moreover by induction on $k$ we see that $\mathrm{End}(W_\lambda^k)$ is exactly given by the functions on $\sym^\lambda \mathbb{A}^1$. With our definition it is straightforward to extend these highest weight structures to a general curve $X$, and to the affine Grassmannian (which may be viewed as an affine analogue of \cite{samokhin2024highest}).

   \begin{theorem}
For a smooth projective curve $X$ we have the following semiorthogonal decomposition
\begin{equation*}
    \D(\bun_G^\xi)=\langle W^k_\lambda \rangle
\end{equation*}
where $W^k_\lambda$ are the global Demazure modules defined above viewed as a functor $W^k_\lambda:\D(\sym^\lambda X)\longrightarrow\D(\bun_G^\xi)$ labelled by $k\in \mathbb{Z}$, and dominant weights $\lambda\in\Lambda^+\simeq\N^{r}$.
\end{theorem}

   \begin{theorem}
We have the following semiorthogonal decomposition for $\gr$:
\begin{equation*}
    \D(\gr)=\langle W^k_\lambda \rangle
\end{equation*}
where $W^k_\lambda$ are the global Demazure modules defined above viewed as a sheaf on $\gr$ labelled by $k\in \mathbb{Z}$, and dominant weights $\lambda\in\Lambda^+\simeq\N^{r}$.
\end{theorem}

The advantage of this semiorthogonal decomposition by global Demazure modules over the decomposition in Theorem 5 by Weyl modules is that it's compatible with the Hecke action, i.e. for fixed $k$, this just corresponds to the action of $\bigoplus_\lambda\D(  \sym^\lambda X)$ on itself as a monoid. Also this gives a canonical way to generate any coherent sheaf by tautological bundles: given any sheaf, we may apply the Hecke operator labelled by $\lambda$, and take the natural (inverse) limit as $\lambda\rightarrow\infty$ in the cone of dominant weights. This limit carries natural $G(\mathcal{K})$ action and can be decomposed into eigensheaves with eigenvalues labelled by levels $G(\mathcal{K})\longrightarrow\B\gm$. Hence we produce canonical maps between tautological generators and any coherent sheaf by Hecke actions.

\section{Reduction to coarse moduli}

In this section we reduce the study of the coarse moduli space of stable bundles $Bun_G^{\xi,ss}$ to that of moduli stack of all bundles $\bun_G^{\xi}$ which we did in previous sections. The main theorem in this section is:

\gnc*

We have natural maps \[
Bun_G^{\xi,ss}\longleftarrow \bun_G^{\xi,ss}/\B Z\longleftarrow \bun_G^{\xi,ss}\longrightarrow \bun_G^{\xi}
\]

The first map (from Artin stack $\bun_G^{\xi,ss}$ to its coarse moduli space) gives a categorical resolution of singularity, which is identity in the smooth case, and this indicates that we should replace $Bun_G^{\xi,ss}$ by $\bun_G^{\xi,ss}/\B Z$ (or its window subcategory). The second map is also easily understood on derived categories, it induces an equivalence $\D(\bun_G^{\xi,ss}/\B Z)\simeq\D (\bun_G^{\xi,ss})^{\B Z}$. So we only need to study the restriction along the third map. Theorem 3 is thus reduced to the following:

\begin{restatable}{theorem}{ss}
Let $X$ be a smooth projective curve of genus $g>1$, $G$ be a semisimple algebraic group of rank $r$, $Z(G)$ be the center of $G$ and $\xi\in H^2(X,Z(G))$, we have a semiorthogonal decomposition for $\bun_{G}^{\xi,ss}$:
\begin{equation*}
    \D(\bun_{G}^{\xi,ss})=\langle \mathcal{L}^{\otimes k}\otimes \D(\sym^{\lambda} X),\cdots\rangle
\end{equation*}
where $\mathcal{L}$ is the fundamental line bundle, $0\leq k\leq \h$, $(g-1)\rho-\lambda$ lies in the cone generated by positive roots, and lies in the interior if $k=\h$. Here $\rho$ is half the sum of positive roots, $h^\vee$ is the dual Coxeter number, and $\cdots$ is the semiorthogonal complement of the previous subcategories.

\end{restatable}

Following \cite{2020arXiv201001127H}, the basic idea is to apply the excision sequence to the internal homomorphism $H=\mathcal{H}om(E,F)\in \D(\bun_G^{\xi})$:

$$\Gamma_{\bun_G^{\xi,us}}(\bun_G^\xi,H)\longrightarrow \Gamma(\bun_G^\xi,H)\longrightarrow\Gamma(\mathrm{Bun}_G^{\xi,ss},H)\longrightarrow$$

When restricted to a suitable subcategory $\mathbf{C}\subset \D(\bun_G^\xi)$ (with bounded weight on isotropy groups in $\mathrm{Bun}_G^{\xi,us}$), the first term vanishes, the restriction functor to the semistable locus is fully-faithful and we get a subcategory of $\D(\mathrm{Bun}_G^{\xi,ss})$.

Let us study when there first term $\Gamma_{\bun_G^{\xi,us}}(\bun_G^\xi,H)=0$. We need to analyze each unstable stratum individually. Note that the Harder-Narasimhan stratification naturally gives rise to maps $Z_G^{\xi,\delta}\overset{\sigma_\delta}{\underset{\pi_\delta}{\rightleftarrows}}\bun_G^{\xi,\delta,us}\overset{i_\delta}{\rightarrow}\bun_G^\xi$ for each Harder-Narasimhan stratum labelled by the parabolic weight $\delta$. On its center $Z_G^{\xi,\delta}$ there is a natural $\B\gm$ action\footnote{$G$ action on a category gives $G$ action on the collection of objects, and $\B G$ action on a category gives $G$ action on each space of morphisms.} coming from $\xi$ with respect to which the weights of $\pi_{\delta*}(\co_{\bun_G^{\xi,\delta,us}})$ are non-negative, so a sufficient condition for $\Gamma_{\bun_G^{\xi,us}}(\bun_G^\xi,H)=0$ would be that the $\gm$ weights of $\sigma_\delta^*i_\delta^!H$ to be strictly positive. It's straightforward calculation that for $F\in \mathcal{L} ^{\otimes k}\otimes \D(\sym^\lambda X) $ and $E\in \mathcal{L} ^{\otimes l}\otimes \D(\sym^\mu X) $, the lowest $\gm$ weight is at least $<2(g-1)\rho-(\lambda^*+\mu),\delta>+(h^\vee-(k-l))|\delta|^2$, where $*$ means the Dynkin involution on the weight lattice coming from the duality in $\mathcal{H}om(E,F)=E^*\otimes F$ , and $h^\vee=h^\vee_G$ is the dual Coxeter number of $G$. So we see that so long as $2(g-1)\rho-(\lambda^*+\mu)$ lies in the positive cone (i.e. cone spanned by the positive roots) and $k-l\leq h^\vee$, and they do not lie on the boundary of either inequalities, the lowest weight is strictly positive, and we have $\Gamma_{\bun_G^{\xi,us}}(\bun_G^\xi,H)=0$. From this we directly see that the composition $\mathcal{L} ^{\otimes k}\otimes \D(\sym^\lambda X) \longrightarrow\D(\bun_G^\xi)\longrightarrow\D(\bun_G^{\xi,ss})$ is fully faithful assuming $2(g-1)\rho-(\lambda^*+\lambda)$ lies in the positive cone, and moreover the semiorthogonal relation of $ \mathcal{L} ^{\otimes k}\otimes \D(\sym^\lambda X) $ and $\mathcal{L} ^{\otimes l}\otimes \D(\sym^\mu X) $ in $\D(\bun_G^\xi)$ is preserved in $\D(\bun_G^{\xi,ss})$ if  $2(g-1)\rho-(\lambda^*+\mu)$ lies in the positive cone. Hence we get Theorem 13. Note that the components of this decomposition are weight subcategories of the $\B Z$ actions, we may directly take the ones with trivial $\B Z$ action, which are exactly those which satisfy the central character of $\mathcal{L}^{\otimes k}$ cancels with the central character of the representation with highest weight $\lambda$.

Note that this is strictly stronger than both  $2(g-1)\rho-(\lambda^*+\lambda)$ and  $2(g-1)\rho-(\mu^*+\mu)$ lying in the positive cone, so we get a number of full subcategories, only a subset of which are semiorthogonal. We need to understand the relation among others. First note that assuming both  $2(g-1)\rho-(\lambda^*+\lambda)$ and  $2(g-1)\rho-(\mu^*+\mu)$ lie in the positive cone, for a given $\delta$, at most one of $<2(g-1)\rho-(\lambda^*+\mu),\delta>, <2(g-1)\rho-(\mu^*+\lambda),\delta>$ can possibly be negative. In other words, there are morphisms at most in one direction supported on each Harder-Narasimhan stratum. This is a local (on $\bun_G^\xi$) version of the semiorthogonal relation among these subcategories, we want to derive a global consequence. 

Given full subcategories $\D(\sym^\lambda X),\D(\sym^\mu X)\subset \D(\bun_G^{\xi,ss})$, we have natural projections $P_{\lambda\mu}$ such that for $E_\lambda\in\D(\sym^\lambda X)$,$E_\mu\in\D(\sym^\mu X)$ we have $\Hom_{\bun_G^{\xi,ss}}(E_\mu,E_\lambda)=\Hom_{\sym^\mu X}(E_\mu,P_{\lambda\mu}(E_\lambda))$. We now show that $P_{\lambda_n\lambda_1}P_{\lambda_{n-1}\lambda_{n}}\cdots P_{\lambda_2\lambda_3}P_{\lambda_1\lambda_2}=0$ for any non-constant chain of weights $\lambda_1\lambda_2\cdots \lambda_n$ such that each $\D(\sym^{\lambda_k} X)\subset \D(\bun_G^{\xi,ss})$ is a full subcategory. From these projection functors it is straightforward to build a semiorthogonal decomposition of $\D(\bun_G^{\xi,ss})$ by the components we want. We can decompose each projection according to its support: from the previous weight argument any map in $\Hom_{\bun_G^{\xi,ss}}(E_\lambda,E_\mu)$ is supported only on $\bun_G^{\xi,\delta}$ such that $\langle2(g-1)\rho-(\lambda^*+\mu),\delta\rangle+h^\vee|\delta|^2\leq 0$, so we get a filtration on the projection functor $P_{\mu\lambda}$ with associated graded $P_{\mu\lambda}^\delta=Q_{\mu\lambda}^\delta R_{\mu\lambda}^\delta$ which naturally factors through the subcategory $\D(\bun_G^{\xi})_{\bun_G^{\xi,\delta}}$, more precisely the subcategory with $\gm$ weight bounded by $\lambda$ as $w\leq\langle \delta,\lambda^*\rangle$. To show  $P_{\lambda_n\lambda_1}P_{\lambda_{n-1}\lambda_{n}}\cdots P_{\lambda_2\lambda_3}P_{\lambda_1\lambda_2}=0$ we only need to show that there exist $i$ such that for all $\delta,\epsilon$, we have $P_{\lambda_i\lambda_{i+1}}^\epsilon P_{\lambda_{i-1}\lambda_{i}}^\delta=0$. So we now need to study when we have $P_{\mu\lambda}^\epsilon P_{\nu\mu}^\delta\neq 0$. By the factorization $P_{\mu\lambda}^\delta=Q_{\mu\lambda}^\delta R_{\mu\lambda}^\delta$, this implies that $R_{\mu\lambda}^\epsilon Q_{\nu\mu}^\delta\neq 0$, and moreover for some $\mathcal{P}\in \bun_G^{\xi,\delta},\mathcal{Q}\in \bun_G^{\xi,\epsilon} $ we have nonvanishing homomorphism $\Hom^*(\co_\mathcal{Q},R_{\mu\lambda}^\epsilon Q_{\nu\mu}^\delta\co_\mathcal{P})\neq 0$ for the skyscraper sheaves twisted by the $\gm$-character, in the category ${\D(\bun_G^{\xi})_{\bun_G^{\xi,\epsilon}}}$. But by adjunction, this homomorphism group is isomorphic to $\Hom^*_{\sym^\lambda X}(\mathcal{Q}^\lambda,\mathcal{P}^\lambda)$ where $S^\lambda=(W_\lambda)^R(\co_\mathcal{P})$ is the global Weyl module on $X$ twisted by the $G$-bundle $\mathcal{P}$.


Note that taking the $\B Z$ invariant subcategory on both sides of Theorem 16, we directly obtain the following:

\begin{restatable}{theorem}{ss2}
Let $X$ be a smooth projective curve of genus $g>1$, $G$ be a semisimple algebraic group of rank $r$, $Z(G)$ be the center of $G$ and $\xi\in H^2(X,Z(G))$, we have a semiorthogonal decomposition for $\bun_{G}^{\xi,ss}$:
\begin{equation*}
    \D(\bun_{G}^{\xi,ss}/\B Z)=\langle \mathcal{L}^{\otimes k}\otimes \D(\sym^{\lambda} X),\cdots\rangle
\end{equation*}
where $\mathcal{L}$ is the fundamental line bundle, $0\leq k\leq \h$, $(g-1)\rho-\lambda$ lies in the cone generated by positive roots, and lies in the interior if $k=\h$, and moreover we have $k\xi+\bar\lambda=0\in Q$. Here $\bar\lambda$ is the equivalence class of $\lambda$ in $Q$, $\rho$ is half the sum of positive roots, $h^\vee$ is the dual Coxeter number, and $\cdots$ is the semiorthogonal complement of the previous subcategories.
\end{restatable}



The excision argument is related naturally to the context of D-modules, we refer the readers to \cite{gaitsgory2019study} for more backgrounds. A prestack over $k$ is a general functor of points, i.e. a functor from finitely generated $k$-algebra to groupoids. For technical reasons (e.g. for limits of categories to be well-behaved) we want to work with DG categories and consider general DG algebras and $\infty$-groupoids, and all categories and functors are automatically derived. Given a prestack $Y$, the category of quasi-coherent sheaf is defined as $$\mathrm{QCoh}(Y):=\lim_{\mathrm{Spec\,}R\rightarrow Y} R\mathrm{-mod}$$

There is also a renormalized version named $\mathrm{IndCoh}(Y)$ which is the ind-completion of coherent sheaves. In our context it makes no difference as all coherent sheaves are perfect for smooth stacks. We use indcoherent sheaf terminology to match with the standard convention in the general context.

The de-Rham prestack is defined as $Y_{dR}(R)=Y(R^{red})$ and the category of D-modules is $\textrm{D-mod}(Y)=\mathrm{IndCoh}(Y_{dR})$. Two basic functors for all maps are $!$-pull $f^!$ and $*$-push $f_*$, and their left adjoints are partially defined. In particular, we have well defined $i_!=i_*$ for closed immersion $Z\longrightarrow Y$, and well defined $j^*=j^!$ for open immersion $U\longrightarrow Y$. If $U=Y \backslash Z$, they form the recollement exact triangle $$i_!i^!\longrightarrow id\longrightarrow j_*j^*\longrightarrow i_!i^![1]$$

Just as coherent cohomology is defined as the push forward along $Y\longrightarrow \mathrm{pt}$, the de Rham cohomology is defined as push forward along $Y_{dR}\longrightarrow \mathrm{pt}$. Moreover we have natural maps $Y\longrightarrow Y_{dR}\longrightarrow \mathrm{pt}$, which gives natural push-forward map $\mathrm{Ind}:\mathrm{IndCoh}(Y)\longrightarrow \textrm{D-mod}(Y)$, and satisfies $\Gamma=\Gamma_{dR}\circ \mathrm{Ind}$. So the calculation of coherent cohomology is reduced to the calculation of de Rham cohomology, and the recollement of D-modules exactly produces the excision triangle of coherent sheaves.

To understand $\Gamma_{\bun_G^{\xi,us}}(\bun_G^\xi,H)$, we only need to understand the de Rham cohomology of $i^!_{\delta}(\ind H)$ along each Harder-Narasimhan stratum $\bun_G^{\xi,\delta,us}\subset \bun_G^\xi$. This global section functor factors through the geometric constant term functor $\ct^{G\to M}:\dmod(\bun_G^\xi)\longrightarrow\dmod(\bun_M^\xi)$ as we have $\Gamma (\bun_G^{\xi,\delta,us},i^!_\delta(\ind H))=\Gamma(\bun_M^{\xi,\delta,ss},j^*\ct^{G\to M}(\ind H))=\Hom(j_*k,\ct^{G\to M}(\ind H))$  where $j=j_\delta:\bun_M^{\delta,\xi,ss}\longrightarrow\bun_M^{\delta,\xi}$ is the open embedding of the semistable loci. Here $\ct^{G\to M}$ is the obvious extension of standard definition of constant term functor to twisted bundles defined as $\ct^{G\to M}:=q_*\circ p^!$ where $\bun_M^\xi \overset{q}{\longleftarrow }\bun_P^\xi \overset{p}{\longrightarrow }\bun_G^\xi$ are the natural maps associated with $M\longleftarrow P\longrightarrow G$ for parabolic subgroup $P\subset G$ and Levi quotient $M$. Using the compatibility between localization and the constant term functor\cite{arinkin2024proof}, we may reduce the calculation of $\Gamma_{\bun_G^{\xi,us}}(\bun_G^\xi,H)$ to semi-infinite cohomology of the Weyl modules in the Kazhdan-Lusztig category and prove the vanishing results we need.

\bibliographystyle{plain}
\bibliography{ref}

@ARTICLE{2020arXiv201001127H,
       author = {{Halpern-Leistner}, Daniel},
        title = "{Derived $\Theta$-stratifications and the $D$-equivalence conjecture}",
      journal = {arXiv e-prints},
         year = 2020,
        month = oct,
          eid = {arXiv:2010.01127},
        pages = {arXiv:2010.01127},
archivePrefix = {arXiv},
       eprint = {2010.01127},
 primaryClass = {math.AG},
       adsurl = {https://ui.adsabs.harvard.edu/abs/2020arXiv201001127H}
}

@article{bennett2011homomorphisms,
  title={On homomorphisms between global Weyl modules},
  author={Bennett, Matthew and Chari, Vyjayanthi and Greenstein, Jacob and Manning, Nathan},
  journal={Representation Theory of the American Mathematical Society},
  volume={15},
  number={24},
  pages={733--752},
  year={2011}
}

@article{sink2025noncommutative,
  title={Noncommutative resolution of SUC (2)},
  author={Sink, Elias and Tevelev, Jenia},
  journal={Advances in Mathematics},
  volume={479},
  pages={110421},
  year={2025},
  publisher={Elsevier}
}

@article{puadurariu2025quasi,
  title={Quasi-BPS categories for symmetric quivers with potential},
  author={P{\u{a}}durariu, Tudor and Toda, Yukinobu},
  journal={Compositio Mathematica},
  volume={161},
  number={8},
  pages={1799--1854},
  year={2025},
  publisher={London Mathematical Society}
}

@article{davison2023bps,
  title={BPS algebras and generalised Kac-Moody algebras from 2-Calabi-Yau categories},
  author={Davison, Ben and Hennecart, Lucien and Mejia, Sebastian Schlegel},
  journal={arXiv preprint arXiv:2303.12592},
  year={2023}
}

@article{tevelev2024bgmn,
  title={The BGMN conjecture via stable pairs},
  author={Tevelev, Jenia and Torres, Sebasti{\'a}n},
  journal={Duke Mathematical Journal},
  volume={173},
  number={18},
  pages={3495--3557},
  year={2024},
  publisher={Duke University Press}
}

@article{arinkin2024proof,
  title={Proof of the geometric Langlands conjecture II: Kac-Moody localization and the FLE},
  author={Arinkin, Dima and Beraldo, Dario and Campbell, J and Chen, L and Faergeman, J and Gaitsgory, Dennis and Lin, K and Raskin, S and Rozenblyum, N},
  journal={arXiv preprint arXiv:2405.03648},
  year={2024}
}

@article{xu2021semiorthogonal,
  title={Semiorthogonal decomposition of $\mathrm{D}^b (\mathrm{Bun}_2^{L}) $},
  author={Xu, Kai and Yau, Shing-Tung},
  journal={arXiv preprint arXiv:2108.13353},
  year={2021}
}

@book{gaitsgory2019study,
  title={A study in derived algebraic geometry: Volume I: correspondences and duality},
  author={Gaitsgory, Dennis and Rozenblyum, Nick},
  volume={221},
  year={2019},
  publisher={American Mathematical Society}
}

@misc{rozenblyum2021connections,
      title={Connections on moduli spaces and infinitesimal Hecke modifications}, 
      author={Nick Rozenblyum},
      year={2021},
      eprint={2108.07745},
      archivePrefix={arXiv},
      primaryClass={math.AG}
}

@incollection{tuy,
  title={Conformal field theory on universal family of stable curves with gauge symmetries},
  author={Tsuchiya, Akihiro and Ueno, Kenji and Yamada, Yasuhiko},
  booktitle={Integrable Sys Quantum Field Theory},
  pages={459--566},
  year={1989},
  publisher={Elsevier}
}

@article{beauville1998picard,
  title={The Picard group of the moduli of-bundles on a curve},
  author={Beauville, Arnaud and Laszio, Yves and Sorger, Christoph},
  journal={Compositio Mathematica},
  volume={112},
  number={2},
  pages={183--216},
  year={1998},
  publisher={London Mathematical Society}
}

@misc{etde_20831639,
title = {Lectures on moduli of principal G-bundles over algebraic curves},
author = {{Sorger}, C},
abstractNote = {These notes are supposed to be an introduction to the moduli of G-bundles on curves. Therefore I will lay stress on ideas in order to make these notes more readable. In the last years the moduli spaces of G-bundles over algebraic curves have attracted some attention from various subjects like from conformal field theory or Beilinson and Drinfeld's geometric Langlands program. In both subjects it turned out that the 'stacky' point of view is more convenient and as the basic motivation of these notes is to introduce to the latter subject our moduli spaces will be moduli stacks (and not coarse moduli spaces). As people may feel uncomfortable with stacks I have included a small introduction to them. Actually there is a forthcoming book of Laumon and Moret-Bailly based on their preprint and my introduction merely does the step -1, i.e. explains why we are forced to use them here and recalls the basic results I need later. So here is the plan of the lectures: after some generalities on G-bundles, I will classify them topologically. Actually the proof is more interesting than the result as it will give a flavor of the basic theorem on G-bundles which describes the moduli stack as a double quotient of loop-groups. This 'uniformization theorem', which goes back to A. Weil as a bijection on sets, will be proved in the section following the topological classification. Then I will introduce two line bundles on the moduli stack: the determinant and the paffian bundle. The first one can be used to describe the canonical bundle on the moduli stack and the second to define a square-root of it. Unless G is simply connected the square root depends on the choice of a theta-characteristic. This square root plays an important role in the geometric Langlands program. Actually, in order to get global differential operators on the moduli stack one has to consider twisted differential operators with values in these square-roots. The rest of the lectures will be dedicated to describe the various objects involved in the uniformization theorem as loop groups or the infinite Grassmannian in some more detail.},
place = {IAEA},
year = {2000},
month = {Aug}
}

@article{teleman1998borel,
  title={Borel-{W}eil-{B}ott theory on the moduli stack of {G}-bundles over a curve},
  author={Teleman, Constantin},
  journal={Inventiones mathematicae},
  volume={134},
  number={1},
  pages={1--57},
  year={1998},
  publisher={Springer}
}

@article{khoroshkin2013highest,
  title={Highest weight categories and Macdonald polynomials},
  author={Khoroshkin, Anton},
  journal={arXiv preprint arXiv:1312.7053},
  year={2013}
}

@article{bennett2014macdonald,
  title={Macdonald polynomials and BGG reciprocity for current algebras},
  author={Bennett, Matthew and Berenstein, Arkady and Chari, Vyjayanthi and Khoroshkin, Anton and Loktev, Sergey},
  journal={Selecta Mathematica},
  volume={20},
  number={2},
  pages={585--607},
  year={2014},
  publisher={Springer}
}

@article{chari2015bgg,
  title={BGG reciprocity for current algebras},
  author={Chari, Vyjayanthi and Ion, Bogdan},
  journal={Compositio Mathematica},
  volume={151},
  number={7},
  pages={1265--1287},
  year={2015},
  publisher={London Mathematical Society}
}

@article{feigin2022duality,
  title={Duality theorems for current groups},
  author={Feigin, Evgeny and Khoroshkin, Anton and Makedonskyi, Ievgen},
  journal={Israel Journal of Mathematics},
  volume={248},
  number={1},
  pages={441--479},
  year={2022},
  publisher={Springer}
}

@article{teleman1995lie,
  title={Lie algebra cohomology and the fusion rules},
  author={Teleman, Constantin},
  journal={Communications in mathematical physics},
  volume={173},
  number={2},
  pages={265--311},
  year={1995},
  publisher={Springer}
}

@article{teleman1996verlinde,
  title={Verlinde factorization and {L}ie algebra cohomology},
  author={Teleman, Constantin},
  journal={Inventiones mathematicae},
  volume={126},
  number={2},
  pages={249--263},
  year={1996},
  publisher={Springer}
}

@book{kumar2012kac,
  title={Kac-Moody groups, their flag varieties and representation theory},
  author={Kumar, Shrawan},
  volume={204},
  year={2012},
  publisher={Springer Science \& Business Media}
}

@incollection{kuznetsov2010derived,
  title={Derived categories of cubic fourfolds},
  author={Kuznetsov, Alexander},
  booktitle={Cohomological and geometric approaches to rationality problems},
  pages={219--243},
  year={2010},
  publisher={Springer}
}

@misc{lee2018remarks,
      title={Remarks on motives of moduli spaces of rank 2 vector bundles on curves}, 
      author={Kyoung-Seog Lee},
      year={2018},
      eprint={1806.11101},
      archivePrefix={arXiv},
      primaryClass={math.AG}
}

@article{polishchuk2019semiorthogonal,
  title={Semiorthogonal decompositions of the categories of equivariant coherent sheaves for some reflection groups},
  author={Polishchuk, Alexander and Van den Bergh, Michel},
  journal={Journal of the European Mathematical Society},
  volume={21},
  number={9},
  pages={2653--2749},
  year={2019}
}

@article{lee2021symmetric,
  title={Symmetric products and moduli spaces of vector bundles of curves},
  author={Lee, Kyoung-Seog and Narasimhan, Mudumbai Seshachalu},
  journal={arXiv preprint arXiv:2106.04872},
  year={2021}
}

@article{drinfeld1995b,
  title={$ {B} $-structures on $ {G} $-bundles and local triviality},
  author={Drinfeld, Vladimir G and Simpson, Carlos},
  journal={Mathematical Research Letters},
  volume={2},
  number={6},
  pages={823--829},
  year={1995},
  publisher={International Press of Boston}
}

@inproceedings{steinberg1962generateurs,
  title={G{\'e}n{\'e}rateurs, relations et rev{\^e}tements de groupes alg{\'e}briques},
  author={Steinberg, Robert},
  booktitle={Colloq. Th{\'e}orie des Groupes Alg{\'e}briques (Bruxelles, 1962)},
  pages={113--127},
  year={1962}
}

@article{tevelev2023braid,
  title={Braid and phantom},
  author={Tevelev, Jenia},
  journal={arXiv preprint arXiv:2304.01825},
  year={2023}
}

@phdthesis{xu2023moduli,
  title={Moduli of vector bundles on curve and semiorthogonal decomposition},
  author={Xu, Kai},
  year={2023},
  school={Harvard University}
}

@article{feigin2004multi,
  title={Multi-dimensional Weyl modules and symmetric functions},
  author={Feigin, Boris and Loktev, Sergei},
  journal={Communications in mathematical physics},
  volume={251},
  number={3},
  pages={427--445},
  year={2004},
  publisher={Springer}
}

@article{chari2010categorical,
  title={A categorical approach to Weyl modules},
  author={Chari, Vyjayanthi and Fourier, Ghislain and Khandai, Tanusree},
  journal={Transformation Groups},
  volume={15},
  number={3},
  pages={517--549},
  year={2010},
  publisher={Springer}
}

@incollection{kac1983regular,
  title={Regular functions on certain infinite-dimensional groups},
  author={Kac, Victor G and Peterson, Dale H},
  booktitle={Arithmetic and Geometry: Papers Dedicated to IR Shafarevich on the Occasion of His Sixtieth Birthday. Volume II: Geometry},
  pages={141--166},
  year={1983},
  publisher={Springer}
}

@inproceedings{tits2006groups,
  title={Groups and group functors attached to Kac-Moody data},
  author={Tits, Jacques},
  booktitle={Arbeitstagung Bonn 1984: Proceedings of the meeting held by the Max-Planck-Institut f{\"u}r Mathematik, Bonn June 15--22, 1984},
  pages={193--223},
  year={2006},
  organization={Springer}
}

@article{samokhin2024highest,
  title={Highest weight category structures on  Rep (B) and full exceptional collections on generalized flag varieties over Z},
  author={Samokhin, Alexander and van der Kallen, Wilberd},
  journal={arXiv preprint arXiv:2407.13653},
  year={2024}
}

@inproceedings{lee2023derived,
  title={Derived category and ACM bundles of moduli space of vector bundles on a curve},
  author={Lee, Kyoung-Seog and Moon, Han-Bom},
  booktitle={Forum of Mathematics, Sigma},
  volume={11},
  pages={e81},
  year={2023}
}

@article{braverman2014weyl,
  title={Weyl modules and q-Whittaker functions},
  author={Braverman, Alexander and Finkelberg, Michael},
  journal={Mathematische Annalen},
  volume={359},
  number={1},
  pages={45--59},
  year={2014},
  publisher={Springer}
}

@article{kato2018demazure,
  title={Demazure character formula for semi-infinite flag varieties},
  author={Kato, Syu},
  journal={Mathematische Annalen},
  volume={371},
  number={3},
  pages={1769--1801},
  year={2018},
  publisher={Springer}
}

@article{zhu2016introduction,
  title={An introduction to affine Grassmannians and the geometric Satake equivalence},
  author={Zhu, Xinwen},
  journal={arXiv preprint arXiv:1603.05593},
  year={2016}
}

@inproceedings{dumanski2021beilinson,
  title={Beilinson--Drinfeld Schubert varieties and global Demazure modules},
  author={Dumanski, Ilya and Feigin, Evgeny and Finkelberg, Michael},
  booktitle={Forum of Mathematics, Sigma},
  volume={9},
  pages={e42},
  year={2021},
  organization={Cambridge University Press}
}

@article{kato2022higher,
  title={Higher level BGG reciprocity for current algebras},
  author={Kato, Syu},
  journal={arXiv preprint arXiv:2207.07447},
  year={2022}
}

@book{marquis2018introduction,
  title={An introduction to Kac-Moody groups over fields},
  author={Marquis, Timoth{\'e}e},
  year={2018},
  publisher={European Mathematical Society Publishing House}
}

@article{kac1985defining,
  title={Defining relations of certain infinite dimensional groups},
  author={Kac, Victor G and Peterson, Dale H.},
  journal={" Elie Cartan et les mathematiques d'aujourd'hui", Lyon, 1984},
  year={1985}
}
\end{document}